\documentclass[secthm,seceqn,amsthm,ussrhead,12pt]{amsart}
\usepackage[english]{babel}

\usepackage{amssymb,amsmath,amsthm,amsfonts,xcolor,enumerate,hyperref,comment,longtable,cleveref}

\usepackage{times}
\usepackage{cite}
\usepackage{pdflscape}
\usepackage{ulem}
\usepackage[mathcal]{euscript}
\usepackage{tikz}
\usepackage{hyperref}
\usepackage{cancel}
\usepackage{stmaryrd}
\usepackage{xcolor}

\usetikzlibrary{arrows}

\newtheorem{thm}{Theorem}
\newtheorem{lem}[thm]{Lemma}
\newtheorem{prop}[thm]{Proposition}

\newtheorem{cor}[thm]{Corollary}

\theoremstyle{definition}
\newtheorem{defn}[thm]{Definition}

\theoremstyle{remark}
\newtheorem{rem}[thm]{Remark}

\usepackage{stmaryrd}
\usepackage{xcolor}

\begin{document}

\noindent{\large
{\bf Some cohomologically rigid solvable Leibniz algebras}}\footnote{
The work was partially supported by Ministerio de Econom\'{i}a y Competitividad (Spain), grant MTM2016-79661-P (European FEDER support included, UE);
RFBR 19-51-04002; 
FAPESP 18/12196-0, 18/12197-7, 18/15712-0.
}

\

   {\bf
            Luisa M. Camacho$^{a}$,
            Ivan  Kaygorodov$^{b}$,
            Bakhrom Omirov$^{c}$ \&
            Gulkhayo Solijanova$^{d}$
 }

{\tiny

\

$^{a}$ Dpto. Matem\'{a}tica Aplicada I. Universidad de Sevilla. Avda. Reina Mercedes, s/n. 41012 Sevilla, Spain.

$^{b}$ CMCC, Universidade Federal do ABC, Santo Andr\'{e}, Brazil.

$^{c}$ Institute of Mathematics of Uzbekistan Academy of Sciences, 81, Mirzo Ulugbek street, Tashkent, 100041, Uzbekistan.

$^{d}$ National University of Uzbekistan, 4, University street, Tashkent, 100174, Uzbekistan.

\

\

\smallskip

   E-mail addresses:

\smallskip

Luisa M. Camacho (lcamacho@us.es)  \smallskip

Ivan   Kaygorodov (kaygorodov.ivan@gmail.com) \smallskip

Bakhrom Omirov (omirovb@mail.ru)  \smallskip

Gulkhayo Solijanova (gulhayo.solijonova@mail.ru)  \smallskip

}

\

{\bf Abstract.}  In this paper we describe solvable Leibniz algebras whose quotient algebra by one-dimensional ideal is a Lie algebra with rank equal to the length of the characteristic sequence of its nilpotent radical. We prove that such Leibniz algebra is unique and centerless. Also it is proved that the first and the second cohomology groups of the algebra with coefficients in itself is trivial.

\medskip


\medskip


\section{Introduction}

Leibniz algebras are characterized
as algebras whose the right multiplication operators are derivations, it is a generalization of Lie algebra, while for a Leibniz algebra to be a Lie algebra it suffices to add the condition that the operators of right and left multiplications alternate. Leibniz algebras have been introduced by Loday in \cite{Lod} as algebras satisfying the (right) Leibniz identity:
$$[x,[y, z]]=[[x, y], z] - [[x, z], y].$$


During the last decades the theory of Leibniz algebras has been actively studied. Some (co)homology and deformation  properties; results on various types of decompositions; structure of solvable and nilpotent Leibniz algebras; classifications of some classes of graded nilpotent Leibniz algebras were obtained in  numerous papers devoted to Leibniz algebras, see, for example,  \cite{AlbAyupov1, Bal, Lod-Pir, Barnes, Gorbat, hnn, geo, deg, hei, qua, ed1, ed2
} and
 reference therein.

In fact, many results on Lie algebras have been extended to the Leibniz algebra case. For instance, an analogue of Levi's theorem for the case of Leibniz algebras asserts that Leibniz algebra is decomposed into a semidirect sum of its solvable radical and a semisimple Lie subalgebra \cite{Barnes}. Therefore, the description of finite-dimensional Leibniz algebras shifts to the study of solvable Leibniz algebras. Since the method of
the reconstruction of solvable Lie algebras from their nilpotent radicals (see \cite{Mub}) was extended to the Leibniz algebras \cite{Nulfilrad}, the main problem of the description of finite-dimensional Leibniz algebras consists of the study of nilpotent Leibniz algebras. Numerous works are devoted to the description of solvable Lie and Leibniz algebras with a given nilpotent radical (see   \cite{Lindsey, BoPaPo, Lisa, Cam, Iqbol, Snobl} and reference therein).

It is known that any Leibniz algebra law can be considered as a point of an affine algebraic variety defined by
the polynomial equations coming from the Leibniz identity for a given basis. This way provides a description of the difficulties in classification problems referring to the classes of nilpotent and solvable Leibniz algebras. The orbits under the base change action of the general linear group correspond to the isomorphism classes of Leibniz algebras therefore, the classification problems (up to isomorphism) can be reduced to the classification of these orbits. An affine algebraic variety is a union of a finite number of irreducible components and the Zariski open orbits provide interesting classes of Leibniz algebras to be classified. The Leibniz algebras of this class are called rigid.

In the study of nilpotent Lie algebras a very useful tool is characteristic sequence, which a priori gives the multiplication on one basis element. Recently, in the paper \cite{Ancochea} it was considered a finite-dimensional solvable Lie algebra $\mathfrak{r}_c$ whose nilpotent radical $\mathfrak{n}_c$ has the simplest structure with a given characteristic sequence $c=(n_1, n_2, \ldots, n_k, 1)$. Using Hochschild -- Serre factorization theorem the authors established that for the algebra $\mathfrak{r}_c$ low order cohomology groups with coefficient in itself are trivial.

In this paper we consider the family of nilpotent Leibniz algebras such that its corresponding Lie algebra is $\mathfrak{n}_c.$ Further, solvable Leibniz algebras with such nilpotent radicals and $(k+1)$-dimensional complementary subspaces to the nilpotent radicals are described. Namely, we prove that such solvable Leibniz algebra is unique and centerless. For this Leibniz algebra the triviality of the first and the second cohomology groups with coefficient in itself is established as well.

\section{Preliminaries}

Throughout the paper, all vector spaces and algebras considered are finite-dimensional over the field of complex numbers $\mathbb{C}$. Moreover, in the table of multiplication of an algebra the omitted products are assumed to be zero.

In this section we give necessary definitions and results on solvable Leibniz algebras and its construction with a given nilpotent radical.

\begin{defn} An algebra $( L,[\cdot,\cdot])$  is called a
Leibniz algebra if it satisfies the property
\begin{center}
\([x,[y,z]]=[[x,y],z] - [[x,z],y]\) for all \(x,y \in
L,\)
\end{center}
which is called Leibniz identity.
\end{defn}

The Leibniz identity is a generalization of the Jacobi identity since under the condition of anti-symmetricity of the product "[$\cdot\, ,\,\cdot$]" this identity changes to the Jacobi identity. In fact, Leibniz algebras is characterized by the property that any right multiplication operator is a derivation.

For a Leibniz algebra $L$, a subspace generated by squares of its elements $\mathcal{I}=\text{span}\left\{[x,x]:  x\in L\right\}$ is a two-sided ideal, and the quotient $\mathcal{G}_L=L/\mathcal{I}$ is a Lie algebra called corresponding Lie algebra (sometimes also called by liezation) of $L.$

For a given Leibniz algebra $L$ we can define the following two-sided ideals
$${\rm Ann}_r(L) =\{x \in L \mid [y,x] = 0,\ \text{for \ all}\ y \in L \},$$
$${\rm Center}(L) =\{x \in L \mid [x,y]=[y,x] = 0,\ \text{for \ all}\ y \in L \}$$
called the {\it right annihilator} and the {\it center} of $L$, respectively.

Applying the Leibniz identity we obtain that for any two elements $x,y$ of an algebra the elements $[x,x], [x,y]+[y,x]$ in ${\rm Ann}_r(L)$.

The notion of a derivation for Leibniz algebras is defined in a usual way and the set of all derivations of $L$ (denoted by ${\mathfrak Der} L $) forms a Lie algebra with respect to the commutator. Moreover, the operator of right multiplication on an element $x\in L$ (further denoted by $\mathcal{R}_x$) is a derivation, which is called {\it inner derivation}.


\begin{defn} A Leibniz algebra $L$ is called \textit{complete} if ${\rm Center}(L)=0$ and all derivations of $L$ are
inner. \end{defn}

For a Leibniz algebra $L$ we define the {\it lower central} and the {\it derived series} as follows:
$$L^1=L, \ L^{k+1}=[L^k,L],  \ k \geq 1, \qquad L^{[1]}=L, \ L^{[s+1]}=[L^{[s]},L^{[s]}], \ s \geq 1,$$
respectively.


\begin{defn} A Leibniz algebra $L$ is called {\it nilpotent} (respectively, {\it solvable}), if there exists $n\in\mathbb N$ ($m\in\mathbb N$) such that $L^{n}=0$ (respectively, $L^{[m]}=0$).
\end{defn}

The maximal nilpotent ideal of a Leibniz algebra is said to be the {\it nilpotent radical} of the  algebra.

Further we shall need the following result from \cite{Ayupov}. It is an extension of the similar result for Lie algebras.
\begin{thm} \label{thmsolv} Let $L$ be a finite-dimensional solvable Leibniz algebra over a field of characteristic zero. Then $L$ is solvable if and only if $L^2$ is nilpotent algebra.
\end{thm}

An analogue of Mubarakzjanov's methods has been applied for solvable Leibniz algebras which shows the importance of the consideration of nilpotent Leibniz algebras and its nil-independent derivations \cite{Nulfilrad}.

\begin{defn} Let $d_1, d_2, \ldots, d_n$ be derivations of a Leibniz algebra $L$. The  derivations $d_1, d_2, \ldots, d_n$ are said to be nil-independent if  $\alpha_1 d_1 + \alpha_2 d_2 + \ldots + \alpha_n d_n$
is not nilpotent for any scalars $\alpha_1, \alpha_2, \ldots, \alpha_n \in \mathbb{C}$, which are not all zero.
%
\end{defn}

In the paper of \cite{Khal} it is proved the following theorem.

\begin{thm}\label{thmLie} Let $R = N \oplus Q$ be a solvable Lie algebra such that $\dim Q =
\dim N/N^2 = k$. Then $R$ admits a basis $\{e_1, e_2, \ldots, e_n, x_1, x_2, \ldots, x_k\}$ such
that the table of multiplication in $R$ has the following form:
$$\left\{\begin{array}{lll}
[e_i, e_j]=\displaystyle \sum_{t=k+1}^{n}\gamma_{i,j}^t e_t,& 1 \leq i, j \leq n,\\[1mm]
[e_i,x_i] = e_i, & 1 \leq i \leq k,\\[1mm]
[e_i,x_j] = \alpha_{i,j}e_i, & k + 1 \leq i \leq n, 1 \leq j \leq k,\\[1mm]
\end{array}\right.$$
where $\alpha_{i,j}$ is the number of entries of a generator basis element $e_j$ involved in forming non generator basis element $e_i$.
\end{thm}

For a nilpotent Leibniz algebra $L$ and $x\in L\setminus L^2$ we consider the decreasing sequence $C(x)=(n_1,n_2,
\ldots,n_k)$ with respect to the lexicographical order of the dimensions Jordan's blocks of the operator $\mathcal{R}_x$.

\begin{defn}  The sequence $C(L)=\max\limits_{x\in L \setminus L^2}C(x)$
is called the {\it characteristic sequence} of the Leibniz algebra $L$.
\end{defn}

In the paper \cite{Ancochea} it is considered the cohomological properties of a solvable Lie algebra whose nilpotent radical has a given characteristic sequence $(n_1, n_2, \ldots, n_k,1)$ and complementary subspace to nilpotent radical has dimension equal to $k+1$.

For characteristic sequence $(n_1, n_2, \ldots, n_k, 1)$ we consider the model nilpotent Lie algebra $\mathfrak{n}_c$ given by its non-zero products:
$$\begin{array}{lll}
[e_i,e_1]=-[e_1,e_i]=e_{i+1}, &2 \leq i \leq n_1, &\\[1mm]
[e_{n_1+\ldots+n_{j}+i},e_1]=-[e_1,e_{n_1+\ldots+n_{j}+i}]=e_{n_1+\ldots+n_{j}+1+i},& 2\leq i\leq n_{j+1}, \ 1\leq j\leq k-1.\\[1mm]
\end{array}$$

Due to Theorem \ref{thmLie} a solvable Lie algebra with nilpotent radical $\mathfrak{n}_c$ and $(k+1)$-dimensional complementary subspace to $\mathfrak{n}_c$ is unique. For our convenience we present its table of multiplication in the following way:
$$\mathfrak{r}_c: \left\{ \begin{array}{lll}
[e_i,e_1]=-[e_1,e_i]=e_{i+1}, &2 \leq i \leq n_1, &\\[1mm]
[e_{n_1+\ldots+n_{j}+i},e_1]=-[e_1,e_{n_1+\ldots+n_{j}+i}]=e_{n_1+\ldots+n_{j}+1+i},& 2\leq i\leq n_{j+1}, \\[1mm]
[e_1,x_1]=-[x_1,e_1]=e_1, &\\[1mm]
[e_i,x_1]=-[x_1,e_i]=(i-2)e_i,& 3\leq i \leq n_1+1,\\[1mm]
[e_{n_1+\ldots+n_{j}+i},x_1]=-[x_1, e_{n_1+\ldots+n_{j}+i}]=(i-2)e_{n_1+\ldots+n_{j}+i} & 2\le i\leq n_{j+1},\\[1mm]
[e_i,x_{2}]=-[x_{2}, e_i]=e_{i}, & 2\le i\leq n_1+1,\\[1mm]
[e_{n_1+\ldots+n_{j}+i},x_{j+2}]=-[x_{j+2}, e_{n_1+\ldots+n_{j}+i}]=
e_{n_1+\ldots+n_{j}+i}, & 2\le i\leq n_{j+1}.\\[1mm]
\end{array}\right.$$
where $1\leq j\leq k-1.$

Here we present the main result of the paper \cite{Ancochea}.
\begin{thm} \label{thmAncochea} For any characteristic sequence $(n_1, \ldots, n_k, 1),$ the model nilpotent Lie algebra $\mathfrak{n}_c$ arises as the nilpotent radical of a solvable Lie algebra $\mathfrak{r}_c$ such that
$$H^a(\mathfrak{r}_c,\mathfrak{r}_c ) = 0, \quad 0 \leq a \leq 3.$$
\end{thm}

\

\subsection{Cohomology of Leibniz algebras}

\

\

We call a vector space $M$ a {\it module over a Leibniz algebra} $L$ if there are two bilinear maps:
$$[-,-] \colon L\times M \rightarrow M \qquad \text{and} \qquad [-,-] \colon M\times L \rightarrow M$$
satisfying the following three axioms
\begin{align*}
[m,[x,y]] & =[[m,x],y]-[[m,y],x],\\
[x,[m,y]] & =[[x,m],y]-[[x,y],m],\\
[x,[y,m]] & =[[x,y],m]-[[x,m],y],
\end{align*}
for any $m\in M$, $x, y \in L$.

For a Leibniz algebra $L$ and module $M$ over $L$ we consider the spaces
$${\rm CL}^0(L,M) = M, \quad {\rm CL}^n(L,M)={\rm Hom} (L^{\otimes n}, M), \ n > 0.$$

Let $d^n : {\rm CL}^n(L,M) \rightarrow {\rm CL}^{n+1}(L,M)$ be an
$\mathbb{C}$-homomorphism defined by
 \begin{multline*}
(d^n\varphi)(x_1, \ldots , x_{n+1}): = [x_1,\varphi(x_2,\ldots,x_{n+1})]
+\sum\limits_{i=2}^{n+1}(-1)^{i}[\varphi(x_1,
\ldots, \widehat{x}_i, \ldots , x_{n+1}),x_i]\\
+\sum\limits_{1\leq i<j\leq {n+1}}(-1)^{j+1}\varphi(x_1, \ldots,
x_{i-1},[x_i,x_j], x_{i+1}, \ldots , \widehat{x}_j, \ldots
,x_{n+1}),
\end{multline*}
where $\varphi\in {\rm CL}^n(L,M)$ and $x_i\in L$. The property $d^{n+1}\circ d^n=0$ leads that the derivative
operator $d=\sum\limits_{i \geq 0}d^i$ satisfies the property
$d\circ d = 0$. Therefore, the $n$-th cohomology group is well defined by
$${\rm HL}^n(L,M): = {\rm ZL}^n(L,M) / {\rm BL}^n(L,M),$$
where the elements
${\rm ZL}^n(L,M):={\rm Ker} \  d^{n+1}$ and ${\rm BL}^n(L,M):={\rm Im} \ d^n$ are called {\it
$n$-cocycles} and {\it $n$-coboundaries}, respectively.

In the case of $n=2$ we give explicit expressions for elements ${\rm ZL}^2(L,L)$ and ${\rm BL}^2(L,L).$ Namely, elements $\psi\in {\rm BL}^2(L,L)$ and $\varphi \in {\rm ZL}^2(L,L)$ are defined by:
\begin{equation}\label{eq4}
\psi(x,y)=[d(x),y] + [x,d(y)] - d([x,y]) \,\, \mbox{for some linear map} \,\, d\in {\rm Hom}(L,L),
\end{equation}
\begin{equation}\label{eq5}
[x,\varphi(y,z)] - [\varphi(x,y), z] +[\varphi(x,z), y] +\varphi(x,[y,z]) - \varphi([x,y],z)+\varphi([x,z],y)=0.
\end{equation}

In terms of cohomology groups the notion of completeness of a Leibniz algebra $L$ means that it is centerless and ${\rm HL}^1(L,L)=0$.

\begin{defn} A Leibniz algebra $L$ is called cohomologically rigid if ${\rm HL}^2(L,L)=0.$
\end{defn}

\begin{rem} \label{rem1} For a centerless Lie algebra $G$ it is known that ${\rm H}^2(G,G)={\rm HL}^2(G,G)$
(see Corollary 2 of \cite{Alice}).
\end{rem}

\section{Main Part}

Let us consider the following family of nilpotent Leibniz algebras $L(\alpha_i, \beta_j)$ with $1\leq i \leq k+1, 1\leq j \leq k$ with a given table of multiplications:
$$\left\{\begin{array}{lll}
[e_i,e_1]=e_{i+1}, &2 \leq i \leq n_1, &\\[1mm]
[e_1,e_i]=-e_{i+1}, &3 \leq i \leq n_1, &\\[1mm]
[e_{n_1+\ldots+n_{j}+i},e_1]=e_{n_1+\ldots+n_{j}+1+i},& 2\leq i\leq n_{j+1}, \ 1\leq j\leq k-1,\\[1mm]
[e_1,e_{n_1+\ldots+n_{j}+i}]=-e_{n_1+\ldots+n_{j}+1+i},& 3\leq i\leq n_{j+1}, \ 1\leq j\leq k-1,\\[1mm]
[e_1,e_1]=\alpha_1h,&  \\[1mm]
[e_2,e_2]=\alpha_2h,&  \\[1mm]
[e_{n_1+\ldots+n_{i}+2},e_{n_1+\ldots+n_{i}+2}]=\alpha_{i+2}h,& 1\leq i\leq k-1. \\[1mm]
[e_1,e_2]=-e_3+\beta_1h,&  \\[1mm]
[e_1,e_{n_1+\ldots+n_{i}+2}]=-e_{n_1+\ldots+n_{i}+3}+\beta_{i+1}h,& 1\leq i\leq k-1, \\[1mm]
\end{array}\right.$$
where $n_1\geq n_2\geq \ldots n_k\geq 1$ and at least one of the parameters $\alpha_i, \beta_j$  is non-zero.



One can assume that $\alpha_1\neq 0.$ Indeed, if $\alpha_1=0$, then taking the following change of the basis
$$e'_1=A_1e_1+A_2e_2+\sum_{i=1}^{k-1}B_ie_{n_1+\ldots+n_{i}+2}, \quad e'_2=e_2, \quad
e'_{i+1}=[e'_i,e'_1], \quad 2 \leq i \leq n_1, \quad \ h'=h,$$
 $$e'_{n_1+\ldots+n_{j}+2}=e_{n_1+\ldots+n_{j}+2}, \quad e'_{n_1+\ldots+n_{j}+1+i}=[e'_{n_1+\ldots+n_{j}+i},e'_1], \quad  2\leq i\leq n_{j+1}, \ 1\leq j\leq k-1,$$
we have
$$[e'_1,e'_1]=(A_2^2\alpha_2+\sum_{i=1}^{k-1}B_i^2\alpha_{i+2}+A_1A_2\beta_1
+A_1\sum_{i=1}^{k-1}B_i^2\alpha_{i+2}\beta_{i+1})h'.$$

 Taking into account that at least one of the parameters $\alpha_i, \beta_j$  is non-zero, we always can chose values $A_1, A_2, B_i$ such that
 $$A_2^2\alpha_2+\sum_{i=1}^{k-1}B_i^2\alpha_{i+2}+A_1A_2\beta_1
+A_1\sum_{i=1}^{k-1}B_i^2\alpha_{i+2}\beta_{i+1}\neq0.$$

Therefore, we can conclude that parameter $\alpha_1$ is non-zero. Now, scaling the basis element $h$ we can assume that $\alpha_1=1$, i.e., $[e_1,e_1]=h$.

Thus, we consider the family of nilpotent Leibniz algebras
$L(\alpha_i, \beta_i)$ with $1\leq i \leq k$ :
$$\left\{\begin{array}{lll}
[e_i,e_1]=e_{i+1}, &2 \leq i \leq n_1, &\\[1mm]
[e_1,e_i]=-e_{i+1}, &3 \leq i \leq n_1, &\\[1mm]
[e_{n_1+\ldots+n_{j}+i},e_1]=e_{n_1+\ldots+n_{j}+1+i},& 2\leq i\leq n_{j+1}, \ 1\leq j\leq k-1,\\[1mm]
[e_1,e_{n_1+\ldots+n_{j}+i}]=-e_{n_1+\ldots+n_{j}+1+i},& 3\leq i\leq n_{j+1}, \ 1\leq j\leq k-1,\\[1mm]
[e_1,e_1]=h,&  \\[1mm]
[e_2,e_2]=\alpha_1h,&  \\[1mm]
[e_{n_1+\ldots+n_{i}+2},e_{n_1+\ldots+n_{i}+2}]=\alpha_{i+1}h,& 1\leq i\leq k-1. \\[1mm]
[e_1,e_2]=-e_3+\beta_1h,&  \\[1mm]
[e_1,e_{n_1+\ldots+n_{i}+2}]=-e_{n_1+\ldots+n_{i}+3}+\beta_{i+1}h,& 1\leq i\leq k-1, \\[1mm]
\end{array}\right.$$
where $n_1\geq n_2\geq \ldots n_k\geq 1$.

\

\subsection{Particular case}\label{subsection31}

\

\

In order to avoid routine calculations which involve many indexes we limit ourselves to the family $L(\alpha_1, \alpha_2, \beta_1, \beta_2)$  with the following table of multiplications:
$$\left\{\begin{array}{lll}
[e_i,e_1]=e_{i+1}, &2 \leq i \leq n_1, &\\[1mm]
[e_1,e_i]=-e_{i+1}, &3 \leq i \leq n_1, &\\[1mm]
[f_{i},e_1]=f_{i+1},& 1\leq i\leq n_2-1,& \\[1mm]
[e_1,f_{i}]=-f_{i+1},& 2\leq i\leq n_2-1,& \\[1mm]
[e_{1},e_1]=h,& [e_{2},e_2]=\alpha_2h,&  [f_{1},f_1]=\alpha_3h,\\[1mm]
[e_{1},e_2]=-e_3+\beta_1h,& [e_{1},f_1]=-f_2+\beta_2h.&\\[1mm]
\end{array}\right.$$

\begin{prop} \label{propder}
	Any derivation of the algebra $L(\alpha_1, \alpha_2, \beta_1, \beta_2)$ has the following matrix form:
$$\mathbb{D}=\begin{pmatrix}
A&B\\
C&D
\end{pmatrix}, \ \mbox{	where}$$
$$\begin{array}{l}
	A=\displaystyle\sum_{j=1}^{n_1+1} \lambda_j e_{1,j}+\displaystyle\sum_{i=2}^{n_1+1} ((i-2)\lambda_1+\gamma_2) e_{i,i}+\displaystyle\sum_{i=2}^{n_1}\sum_{j=i+1}^{n_1+1} \gamma_{j-i+2} e_{i,j},\quad
	C=\displaystyle\sum_{i=1}^{n_2}\sum_{j=i+1}^{n_1+1} \theta_{j-i+1} e_{i,j},\\
	B=\displaystyle\sum_{j=1}^{n_2} \mu_j e_{1,j}+\displaystyle\sum_{i=1}^2 c_i e_{i,n_2+1}+(\lambda_2\alpha_1)e_{3,n_2+1}+\displaystyle\sum_{i=2}^{n_2}\sum_{j=i-1}^{n_2} \delta_{j-i+2} e_{i,j},\\
	D=\displaystyle\sum_{i=1}^{n_2} ((i-1)\lambda_1+\nu_1) e_{i,i}+c_3 e_{1,n_2+1}+\displaystyle\sum_{i=2}^{n_2}\sum_{j=i+1}^{n_2} \nu_{j-i+1} e_{i,j}+(\mu_1\alpha_2)e_{2,n_2+1}+m e_{n_2+1,n_2+1}
	\end{array}$$
$m=(2\lambda_1+\lambda_2\beta_1+\mu_1\beta_2),$ $A\in M_{n_1+1,n_1+1},$ $B\in M_{n_1+1,n_2+1},$ $C\in M_{n_2+1,n_1+1},$
$D\in M_{n_2+1,n_2+1}$ and matrix units $e_{i,j}$ and with the restrictions:
$$\left\{
\begin{array}{lll}
\alpha_1\theta_2+\alpha_2\delta_1=0, \\[1mm]
-2\lambda_2\alpha_1+\lambda_1\beta_1+\lambda_2\beta_1^2+\mu_1\beta_1\beta_2-\gamma_2\beta_1-\delta_1\beta_2=0,&\\[1mm]
-2\mu_1\alpha_2+\lambda_1\beta_2+\lambda_2\beta_1\beta_2+\mu_1\beta_2^2-\theta_2\beta_1-\nu_1\beta_2=0,&\\[1mm]
\alpha_1(2\lambda_1+\lambda_2\beta_1+\mu_1\beta_2-2\gamma_2)=0,&\\[1mm]
\alpha_2(2\lambda_1+\lambda_2\beta_1+\mu_1\beta_2-2\nu_1)=0.&\\[1mm]
\end{array}\right.
$$
Moreover, if $n_1 > n_2$, then $\theta_j=0$ with $2\leq j\leq n_1-n_2+1.$
\end{prop}
\begin{proof} The proof is carried out by straightforward checking the derivation property and using the table
of multiplications of the algebras $L(\alpha_1, \alpha_2, \beta_1, \beta_2).$
\end{proof}

\begin{lem} \label{lem1} Let $d$ be a derivation of the algebra $L(\alpha_1, \alpha_2, \beta_1, \beta_2)$. Then we have that coefficient  $d(h)|_{h}$ is 
$\epsilon_1+\epsilon_2,$ where $\epsilon_k \in \{  \nu_1, \lambda_1, \gamma_2 \}.$
\end{lem}
\begin{proof} Let us consider the following cases:
	\begin{enumerate}
	\item $\alpha_2\neq 0$. In this case, by applying the derivation conditions we have $2\lambda_1+\lambda_2\beta_1+\mu_1\beta_2-2\nu_1=0$ then $d(h)=2\nu_1h.$
	\item $\alpha_2=0$ and $\alpha_1\neq 0.$ Similar to the above case we have $d(h)=2\gamma_2h.$
	\item $\alpha_2=0$ and $\alpha_1= 0.$ We consider the following:
	\begin{enumerate}
		
		\item $\beta_2\neq 0$. Making the following change of basis:
		$e'_1=e_1,\ e'_i=\beta_2e_i-\beta_1 f_{i-1},\ 2\leq i\leq n_2+1, $ $e_i=\beta_2 e_i, n_2+2\leq i\leq n_1+1$ and $f'_i=f_i,\ 1\leq i\leq n_2,$ we can suppose $\beta'_1=0$ and by restrictions we have that $\mu_1\beta_2=\nu_1-\lambda_1.$ Hence, $d(h)=(\lambda_1+\nu_1)h.$

	\item $\beta_2=\beta_1=0.$ Then $d(h)=2\lambda_1h.$
\item $\beta_2=0,\ \beta_1\neq 0.$ By restrictions we have that $\lambda_2\beta_1=\gamma_2-\lambda_1.$ Therefore, $d(h)=(\lambda_1+\gamma_2)h.$

	\end{enumerate}
	\end{enumerate}

\end{proof}

\begin{lem} \label{lem2} The number of nil-independent derivations of the algebra $L(\alpha_1, \alpha_2, \beta_1, \beta_2)$ is equal to 4.
\end{lem}

\begin{proof} We are going to prove that the matrix $\mathbb{D}$ is a nilpotent matrix if and only if $\lambda_1=\gamma_2=\nu_1=\delta_1\theta_2=0.$ By Lemma \ref{lem1}, we have that $d(h)=(a_1\lambda_1+a_2\gamma_2+\alpha_3 \nu_1)h.$

According Proposition \ref{propder} we have
$$\mathbb{D}=\left(\begin{array}{cc}
A&B\\
C&D\\
\end{array}\right)=\left(\begin{array}{cc}
A_1+A_2&B\\
C&D_1+D_2\\
\end{array}\right), \ \mbox{where}$$
$$\begin{array}{l}
	 	 A_1=diag\{\lambda_1,\gamma_2,\lambda_1+\gamma_2,\ldots,(n_1-1)\lambda_1+\gamma_2\},\\[1mm] D_1=diag\{\nu_1,\lambda_1+\nu_1,2\lambda_1+\nu_1,\ldots, (n_2-1)\lambda_1+\nu_1, a_1\lambda_1+a_2\gamma_2+\alpha_3 \nu_1\}
	 	 \end{array}$$ are diagonal matrices, $A_2, D_2, C$ are strictly upper triangular matrices and the matrix $B$ is upper triangular matrix with non-zero diagonal under the main diagonal such that $A\in M_{n_1+1,n_1+1},$ $B\in M_{n_1+1,n_2+1},$ $C\in M_{n_2+1,n_1+1}$ and $D\in M_{n_2+1,n_2+1}.$

Note that matrices $A_1A_2, \ A_2^2, \ D_1D_2, \ D_1^2$ are nilpotent, the matrices $C(A_1+A_2)$ and $(D_1+D_2)C$ have the same type pattern as $C$ (that is if any entry of $C$ is $0,$ the entry of $C(A_1+A_2)$ and the entry of  $(D_1+D_2)C$ at the same position is zero, as well). Likewise, the matrix $(A_1+A_2)B$ and $ B(D_1+D_2)$ has the same pattern as $B$.

It is easy to see that $BC=K_1+K_2$ with diagonal matrix $K_1=diag\{\underbrace{0,\delta_1\theta_2,\delta_1\theta_2,\ldots,\delta_1\theta_2}_{n_2},0,\ldots,0\}$  and strictly upper triangular matrix $K_2.$ Similarly, $CB=Z_1+Z_2$ with diagonal matrix $Z_1=diag\{\delta_1\theta_2,\delta_1\theta_2,\ldots,\delta_1\theta_2,0\}$ and strictly upper triangular matrix $Z_2.$
	
According to the above arguments we have the following  formula:
	$$\mathbb{D}^2=\left(\begin{array}{cc}
	\widetilde{A}_1+\widetilde{A}_{2}&\widetilde{B}\\
	\widetilde{C}&\widetilde{D}_1+\widetilde{D}_2\\
	\end{array}\right), $$
where $\widetilde{A}_{2}, \ \widetilde{D}_2-$ nilpotent matrices  and the matrices $\widetilde{B} $ and $ \widetilde{C}$ are the same type as $B$ and $C,$ respectively and $\widetilde{A}_1$ and $\widetilde{D}_1$ are the  following diagonal matrices:
	$$\begin{array}{l}
	\widetilde{A}_1=diag\{\lambda_1^2,\gamma_2^2+\delta_1\theta_2,(\lambda_1+\gamma_2)^2+\delta_1\theta_2,
\ldots,((n_2-2)\lambda_1+\gamma_2)^2+\delta_1\theta_2,\\
\qquad \quad\ \quad ((n_2-1)\lambda_1+\gamma_2)^2,\ldots,       ((n_1-1)\lambda_1+\gamma_2)^2\},\\[2mm]
\widetilde{D}_1=diag\{\nu_1^2+\delta_1\theta_2,(\lambda_1+\nu_1)^2+\delta_1\theta_2,(2\lambda_1+\nu_1)^2+\delta_1\theta_2,\ldots,\\ \qquad\ \quad\ \quad \ldots, ((n_2-1)\lambda_1+\nu_1)^2+\delta_1\theta_2, (a_1\lambda_1+a_2\gamma_2+\alpha_3 \nu_1)^2\}.
\end{array}$$

To continue iteration we conclude that in the main diagonal of the matrix $\mathbb{D}^k$ will be equal to zero if and only if $\lambda_1=\gamma_2=\nu_1=\delta_1\theta_2=0.$ Thus, the nilpotency of the matrix $\mathbb{D}$ implies $\lambda_1=\gamma_2=\nu_1=\delta_1\theta_2=0.$

Let us assume now that $\lambda_1=\gamma_2=\nu_1=\delta_1\theta_2=0$. Then we obtain that matrices $\widetilde{A}_1+\widetilde{A}_{2}, \ \widetilde{D}_1+\widetilde{D}_{2}, \widetilde{C}$ are strictly upper triangular and the matrix $\widetilde{B}$ is upper triangular. Therefore, the matrix $\mathbb{D}^2$ is nilpotent and hence, $\mathbb{D}$ is nilpotent.
\end{proof}

Let $R$ be a solvable Leibniz algebra whose nilpotent radical is the algebras from $L(\alpha_1, \alpha_2, \beta_1, \beta_2).$ We denote by $Q$ the complementary subspace to a nilpotent radical of $R$. Due to work \cite{Nulfilrad} we have that
dimension of $Q$ is bounded by number of nil-independent derivations of $L(\alpha_1, \alpha_2, \beta_1, \beta_2).$

Let us introduce denotations
$$d_1\in {\mathfrak Der} L(\alpha_1, \alpha_2, \beta_1, \beta_2) \ \mbox{with} \ \lambda_1\neq 0, \ \gamma_2=\nu_1=\delta_1\theta_2=0,$$
$$d_2\in {\mathfrak Der} L(\alpha_1, \alpha_2, \beta_1, \beta_2) \ \mbox{with} \ \gamma_2\neq 0, \ \lambda_1=\nu_1=\delta_1\theta_2=0,$$
$$d_3\in {\mathfrak Der} L(\alpha_1, \alpha_2, \beta_1, \beta_2) \ \mbox{with} \ \nu_1\neq 0, \ \lambda_1=\gamma_2=\delta_1\theta_2=0,$$
$$d_4 \in {\mathfrak Der}  L(\alpha_1, \alpha_2, \beta_1, \beta_2) \ \mbox{with} \ \delta_1\theta_2\neq 0, \ \lambda_1=\gamma_2=\nu_1=0.$$

\begin{prop} \label{cor1} $\dim Q\leq 3$.
\end{prop}
\begin{proof} Due to Lemma \ref{lem2} we have that the number of nil-independent derivations of $L(\alpha_1, \alpha_2, \beta_1, \beta_2)$ is equal to 4 and they are depends on parameters $\lambda_1, \gamma_2, \nu_1, \theta_2, \delta_1$. Let us assume that $\dim Q=4,$ that is, $Q=\{x_1, x_2, x_3, x_4\}$. Then
$${\mathcal{R}_{x_i}}_{|L(\alpha_1, \alpha_2, \beta_1, \beta_2)}=d_i, \ i=1, \ldots, 4.$$

By scaling of the basis elements $x_i, \ 1\leq i \leq 4$ one can assume that $\lambda_1=1$ in $d_1,$ $\gamma_2=1$ in $d_2$, $\nu_1=1$ in $d_3$, respectively.

Let us assume that $\theta_2\neq 0$ (recall that this case is impossible when $n_1>n_2$).
Thanks to Theorem \ref{thmsolv} we have $R^2\subseteq L(\alpha_1, \alpha_2, \beta_1, \beta_2).$
Applying this embedding in the following equalities:
$$(*)f_2=[f_1,[x_2,x_4]]= [[f_1,x_2],x_4]-[[f_1,x_4],x_2]=-e_2+L(\alpha_i, \beta_j)^2$$
we get a contradiction with the assumption that $\theta_2\neq 0$. Thus, we obtain $\dim Q\leq 3.$
\end{proof}

The following theorem describes solvable Leibniz algebras with nilpotent radical $L(\alpha_1, \alpha_2, \beta_1, \beta_2)$
and maximal possible dimension of $Q$.

\begin{thm} \label{thmDescription}Solvable Leibniz algebra with nilpotent radical $L(\alpha_1, \alpha_2, \beta_1, \beta_2)$ and three-dimensional complementary subspace is isomorphic to the algebra:
$$R:\left\{ \begin{array}{llll}
[e_{1},e_1]=h,&[e_i,e_1]=-[e_1,e_i]=e_{i+1}, &2 \leq i \leq n_1,\\[1mm]
[h,x_1]=2h,&[f_{i},e_1]=-[e_1,f_{i}]=f_{i+1},& 1\leq i\leq n_2-1,\\[1mm]
[e_{1},x_1]=-[x_1,e_1]=e_1,& [e_i,x_1]=-[x_1,e_i]=(i-2)e_i, & 3\leq i\leq n_1+1,&\\[1mm]
[f_i,x_1]=-[x_1,f_i]=(i-1)f_i, & 2\leq i\leq n_2,\\[1mm]
[e_i,x_2]=-[x_2,e_i]=e_i, & 2\leq i\leq n_1+1,& &\\[1mm]
[f_i,x_3]=-[x_3,f_i]=f_i, & 1\leq i\leq n_2.& &\\[1mm]
\end{array}\right.$$
\end{thm}
\begin{proof} Let $R=L(\alpha_1, \alpha_2, \beta_1, \beta_2)\oplus Q$ with $\{e_1,e_2,\ldots , e_{n_1+1}, f_1\ldots, f_{n_2}, x_1, x_2, x_3\}$ such that
$${\mathcal{R}_{x_i}}_{|L(\alpha_1, \alpha_2, \beta_1, \beta_2)}=d_i, \ i=1,2,3.$$
Due to Proposition \ref{propder} we have the products $[L(\alpha_1, \alpha_2, \beta_1, \beta_2),x_j], 1\leq j \leq 3.$

From the table of multiplications of the algebra $L(\alpha_1, \alpha_2, \beta_1, \beta_2)$ we derive that
$$e_i\not \in {\rm Ann}_r(R) \ \ \mbox{with} \ \ 1\leq i\leq n_1, \quad f_j\not \in {\rm Ann}_r(R) \ \ \mbox{with} \ 1\leq j\leq n_2-1.$$ Taking into account that for any $x,y\in R$ we have $[x,y]+[y,x]\in Ann_r(R)$ we conclude that
$$\begin{array}{ll}
[x_1,e_{i}] +[e_{i},x_1]=(*)e_{n_1+1}+(*)f_{n_2}+(*)h, & 1\leq i \leq n_1,\\[1mm]
[x_1,f_{i}]+[f_{i},x_1]=(*)e_{n_1+1}+(*)f_{n_2}+(*)h, & 1\leq i \leq n_2-1.
\end{array}$$

Consider
$$\begin{array}{lll}
[x_1,e_{n_1+1}]&=&[x_1,[e_{n_1},e_1]]=[[x_1,e_{n_1}],e_1]-[[x_1,e_1],e_{n_1}]=\\[1mm]
&=&[-[e_{n_1},x_1],e_1]-[-[e_1,x_1],e_{n_1}]=-(n_1-1)e_{n_1+1},\\[1mm]
[x_1,f_{n_2}]&=&[x_1,[f_{n_2-1},e_1]]=[[x_1,f_{n_2-1}],e_1]-[[x_1,e_1],f_{n_2-1}]=\\[1mm]
&=&[-[f_{n_2-1},x_1],e_1]-[-[e_1,x_1],f_{n_2-1}]=\\[1mm]
&-&\displaystyle\sum_{k=n_2+2}^{n_1+1} \theta_{k-n_2+1,1} e_k-(n_2-1)f_{n_2}.
\end{array}$$

 This imply that $e_{n_1+1}, f_{n_2}\notin {\rm Ann}_r(R)$

We claim that $span \langle e_i, f_j \ | 1\leq i \leq n_1+1, \ 1\leq j \leq n_2 \rangle \cap {\rm Ann}_r(R)=\{0\}$.
Indeed, let
$$z=a_1e_1+a_2e_2+\ldots+a_{n_1+1}e_{n_1+1}+b_1f_1+\ldots+b_{n_2}f_{n_2}+c_1x_1+c_2x_2+c_3x_3\in {\rm Ann}_r(R).$$ Then considering the products
$$0=[x_1,z]=[x_2,z]=[x_3,z]=[e_1,z]=[e_2,z]=[f_1,z]$$
we derive $z=0.$

Thus, we obtain ${\rm Ann}_r(R)= \langle h \rangle$ and
$$\begin{array}{lll}
[x_j,e_i]=-[e_i,x_j]+(*)h, & 1\leq i \leq n_1+1, &1\leq j \leq 3,\\[1mm]
[x_j,f_i]=-[f_i,x_j]+(*)h, & 1\leq i \leq n_2,& 1\leq j \leq 3,\\[1mm]
[x_i,x_j]=-[x_j,x_i]+(*)h, & 1\leq i,j\leq 3.&
\end{array}$$

It is easy to see that the quotient algebra $R/{\rm Ann}_r(R)$ is a particular case of the Lie algebra $\mathfrak{r}_c.$
Namely, the quotient Lie algebra has nilpotent radical $\mathfrak{n}_c$ with characteristic sequence $(n_1, n_2, 1)$ and its table of multiplication has the following form:
$$\left\{\begin{array}{llll}
[e_i,e_1]=-[e_1,e_i]=e_{i+1}, &2 \leq i \leq n_1, &\\[1mm]
[f_{i},e_1]=-[e_1,f_{i}]=f_{i+1},& 1\leq i\leq n_2-1,&\\[1mm]
[e_{1},x_1]=-[x_1,e_1]=e_1,& \\[1mm]
[e_i,x_1]=-[x_1,e_i]=(i-2)e_i, & 3\leq i\leq n_1+1,&\\[1mm]
[f_i,x_1]=-[x_1,f_i]=(i-1)f_i, & 2\leq i\leq n_2,&\\[1mm]
[e_i,x_2]=-[x_2,e_i]=e_i, & 2\leq i\leq n_1+1,&\\[1mm]
[f_i,x_3]=-[x_3,f_i]=f_i, &\  1\leq i\leq n_2.&&\\[1mm]
\end{array}\right.$$

If we now rise up to the initial algebra $R$, then we get the following table of multiplications (we omit the bracket of the family $L(\alpha_1, \alpha_2, \beta_1, \beta_2)$):
$$\begin{array}{ll}
[e_1,x_j]=\delta_{1j}e_1+ a_j h,&[x_j,e_1]=-\delta_{1j}e_1+\widetilde{a}_j h,\ 1\leq j\leq 3\\[1mm]
[e_2,x_j]=\delta_{2j}e_2+b_j h,&[x_j,e_2]=-\delta_{2j}e_2+\widetilde{b}_j h,\ 1\leq j\leq 3 \\[1mm]
[e_3,x_j]=(1-\delta_{3,j})e_3+ c_j  h ,& 1\leq j\leq 3\\[1mm]
[e_i,x_1]=(i-2)e_i,& 4\leq i\leq n_1+1,\\[1mm]
[e_i,x_2]=e_i,& 4\leq i\leq n_1+1,\\[1mm]

[f_1,x_j]=\lambda f_1+ d_j h,& [x_j,f_1]=-\lambda f_1+ \widetilde{d}_1 h,\\[1mm]
[f_2,x_j]=(1-\delta_{2j}) f_2+ g_{j} h,& \\[1mm]
[f_i,x_1]=(i-1)f_i,& 3\leq i\leq n_2,\\[1mm]
[f_i,x_3]=f_i,&  3\leq i\leq n_2,\\[1mm]

[h,x_j]=m_j h, &   1\leq j\leq 3\\[1mm]
[x_i,x_j]=\varphi_{i,j}h,& 1\leq i,j\leq 3.
\end{array}$$
with $\delta_{ij}$ the Kronecker symbol, $\lambda=0$ if $ j=1,2$ and $\lambda =1$ if $j=3.$

The Leibniz identity on the following triples imposes further constraints on the above family.
\begin{center}
	\begin{tabular}{ll}
	\begin{tabular}{lll}
	Leibniz identity& &Constraint\\
		\hline\hline
		$\{e_1,e_1,x_1\},$ &\quad $\Rightarrow $\quad & $m_1=2,$\\
		$\{e_1,e_1,x_i\},\ 2\leq i\leq 3$ &\quad $\Rightarrow $\quad & $m_i=0,$\\
		$\{e_2,e_2,x_1\},$ &\quad $\Rightarrow $\quad & $\alpha_1=0,$\\
		$\{f_1,f_1,x_1\},$ &\quad $\Rightarrow $\quad & $\alpha_2=0,$
	\end{tabular}&	\begin{tabular}{lll}
	Leibniz identity& &Constraint\\
	\hline\hline
	$\{e_1,e_2,x_1\},$ &\quad $\Rightarrow $\quad & $\beta_1=0,$\\
	$\{e_1,f_1,x_1\},$ &\quad $\Rightarrow $\quad & $\beta_2=0,$\\
	$\{e_2,e_1,x_i\},\ 1\leq i\leq 3$ &\quad $\Rightarrow $\quad & $c_i=0,$\\
	$\{f_1,e_1,x_i\},\ 1\leq i\leq 3$ &\quad $\Rightarrow $\quad & $g_i=0.$

\end{tabular}
\end{tabular}
\end{center}
At that time, the following change of basis
$$e_1'=e_1-a_1h,\quad e_2'=e_2-\frac{b_1}{2}h,\quad f_1'=f_1-\frac{d_1}{2}h,\quad x_i'=x_i-\frac{\varphi_{i,1}}{2}h,\ 1\leq i\leq 3$$
allows to assume that
 $a_1=b_1=d_1=\varphi_{i,1}=0$ for $ 1\leq i\leq 3.$

We again apply the Leibniz identity and we have the following:
\begin{center}
	\begin{tabular}{lll}
		Leibniz identity& &Constraint\\
		\hline\hline
		$\{x_i,e_1,x_1\},\ 1\leq i\leq 3$ &\quad $\Rightarrow $\quad & $\widetilde{a}_i=0,\ 1\leq i\leq 3,$\\
		$\{x_i,e_2,x_1\},\ 1\leq i\leq 3$ &\quad $\Rightarrow $\quad & $\widetilde{b}_i=0,\ 1\leq i\leq 3,$\\
			$\{f_1,x_i,x_1\},\ 2\leq i\leq 3$ &\quad $\Rightarrow $\quad & $d_i=0,\ 2\leq i\leq 3,$\\
		$\{x_i,f_1,x_1\},\ 2\leq i\leq 3$ &\quad $\Rightarrow $\quad & $\widetilde{d}_i=0,\ 1\leq i\leq 3,$\\
		$\{e_1,x_i,x_1\},\ 2\leq i\leq 3$ &\quad $\Rightarrow $\quad & $a_i=0,\ 2\leq i\leq 3,$\\
		$\{e_2,x_i,x_1\},\ 2\leq i\leq 3$ &\quad $\Rightarrow $\quad & $b_i=0,\ 2\leq i\leq 3,$\\
		$\{x_i,x_j,x_1\},\ 1\leq i\leq 3,\ 2\leq j\leq 3$ &\quad $\Rightarrow $\quad & $\varphi_{i,j}=0,\ 1\leq i\leq 3,\ 2\leq j\leq 3.$
	\end{tabular}
\end{center}

Finally, if  we consider the equalities $[x_j,e_i]=[x_j,[e_{i-1},e_1]]$  with $3\leq i\leq n_1+1 $ and $[x_j,f_i]=[x_j,[f_{i-1},e_1]]$ with $2\leq i\leq n_2,\ 1\leq j\leq 3 $ we obtain the algebra of the theorem statement.

\end{proof}

The next result establish the completeness of the algebra $R$.

\begin{thm} \label{thmcomplete} The solvable Leibniz algebra $R$ is complete.
\end{thm}
\begin{proof} Centerless of the algebra $R$ is immediately follows from the table of multiplications in Theorem \ref{thmDescription}. Note that $\langle h \rangle$ forms an ideal of $R$.

The quotient algebra $R/ \langle h \rangle$ is the algebra ${\mathfrak{r}_c}$, which is complete due to Theorem \ref{thmAncochea}. Applying this result in the following equalities by modulo of an ideal $\langle h \rangle$:
$$\begin{array}{rll}
d(e_1)&=&d([e_1,x_1])=[d(e_1),x_1]+[e_1,d(x_1)] \not\equiv  0 \Rightarrow d(e_1)\equiv \mathcal{R}_{\alpha e_1}(e_1)=\alpha h, \ \alpha\in \mathbb{C},\\
0&=&d([e_2,x_1])=[d(e_2),x_1]+[e_2,d(x_1)]\equiv [d(e_2),x_1]\Rightarrow d(e_2) \equiv 0,\\
d(e_{i+1})&=&d([e_i,e_1])=[d(e_i),e_1]+[e_i,d(e_1)]\equiv 0 \Rightarrow d(e_{i+1}) \equiv 0, \ 2\leq i \leq n_1,\\
0&=&d([f_1,x_1])=[d(f_1),x_1]+[f_1,d(x_1)]\equiv [d(f_1),x_1] \Rightarrow d(f_1) \equiv 0,\\
d(f_{i+1})&=&d([f_i,e_1])=[d(f_i),e_1]+[f_i,d(e_1)]\equiv 0 \Rightarrow d(f_{i+1}) \equiv 0, \ 1\leq i \leq n_2-1,\\
0&=&d([x_i,x_1])=[d(x_i),x_1]+[x_i,d(x_1)]\equiv [d(x_i),x_1] \Rightarrow d(x_i) \equiv 0, 1\leq i \leq 3.
\end{array}$$
and in the chain of equalities
$$\begin{array}{ll}
2[d(e_1),e_1]+2[e_1,d(e_1)]&=2d([e_1,e_1])=2d(h)=d([h,x_1]) =[d(h),x_1]+[h,d(x_1)] \\
&\Rightarrow d(h) = \mathcal{R}_{\beta x_1}(h)=2\beta h, \ \beta\in \mathbb{C}
\end{array}$$
we conclude that any derivation of $R$ is inner.

\end{proof}


Now we prove the triviality of the second group of cohomology for the algebra $R$ with coefficient itself (that is ${\rm HL}^2(R,R)=0$). Since $J:=\langle h \rangle $ is an ideal of $R$ and quotient algebra $R/J$ is the Lie algebra $\mathfrak{r}_c$, we get a decomposition $R=\mathfrak{r}_c\oplus J$ as the direct sum of the vector spaces (here we identify the space of the quotient space ${\mathfrak{r}_c}$ and its preimage under the natural homomorphism). Hence, for any $x,y\in R$ and $\varphi(x,y)\in {\rm ZL}^2(R,R)$ one has
$$[x,y]=[x,y]_{\mathfrak{r}_c}+[x,y]_J, \quad \varphi(x,y)=\varphi(x,y)_{\mathfrak{r}_c}+\varphi(x,y)_J,$$
with $[x,y]_{\mathfrak{r}_c}\in \mathfrak{r}_c, \  [x,y]_J\in J$ and $\varphi(x,y)_{\mathfrak{r}_c}\in {\mathfrak{r}_c}, \ \varphi(x,y)_J\in J$.

For an arbitrary elements $x,y,z \in \mathfrak{r}_c$ and $\varphi \in {\rm ZL}^2(R,R)$ using (\ref{eq4}) we consider the chain of equalities:
$$\begin{array}{ll}
0&=[x,\varphi(y,z)]-[\varphi(x,y),z]+[\varphi(x,z),y]+\varphi(x,[y,z])-\varphi([x,y],z)+
\varphi([x,z],y)=\\{}
%
%
%
&=[x,\varphi(y,z)_{\mathfrak{r}_c}]_{\mathfrak{r}_c}-[\varphi(x,y)_{\mathfrak{r}_c},z]_{\mathfrak{r}_c}
+[\varphi(x,z)_{\mathfrak{r}_c},y]_{\mathfrak{r}_c}+\varphi(x,[y,z]_{\mathfrak{r}_c})_{\mathfrak{r}_c}
-\varphi([x,y]_{\mathfrak{r}_c},z)_{\mathfrak{r}_c}+\\{}
&+\varphi([x,z]_{\mathfrak{r}_c},y)_{\mathfrak{r}_c}+[x,\varphi(y,z)_{\mathfrak{r}_c}]_J+[x,\varphi(y,z)_J]_J-
[\varphi(x,y)_{\mathfrak{r}_c},z]_J-[\varphi(x,y)_J,z]_J+\\{}
&+[\varphi(x,z)_{\mathfrak{r}_c},y]_J+[\varphi(x,z)_J,y]_J+
\varphi(x,[y,z]_{\mathfrak{r}_c})_J+\varphi(x,[y,z]_J)_{\mathfrak{r}_c}+\varphi(x,[y,z]_J)_J-\\{}
&-\varphi([x,y]_{\mathfrak{r}_c},z)_J-\varphi([x,y]_J,z)_{\mathfrak{r}_c}-\varphi([x,y]_J,z)_J+
\varphi([x,z]_{\mathfrak{r}_c},y)_J+\varphi([x,z]_J,y)_{\mathfrak{r}_c}+\\{}
&+\varphi([x,z]_J,y)_J.
\end{array}$$

From this we obtain
\begin{equation}\label{eq111}\left\{
\begin{array}{lll}
[x,\varphi(y,z)_{\mathfrak{r}_c}]_{\mathfrak{r}_c}-[\varphi(x,y)_{\mathfrak{r}_c},z]_{\mathfrak{r}_c}+
[\varphi(x,z)_{\mathfrak{r}_c},y]_{\mathfrak{r}_c}+&\\[1mm]
\varphi(x,[y,z]_{\mathfrak{r}_c})_{\mathfrak{r}_c}-\varphi([x,y]_{\mathfrak{r}_c},z)_{\mathfrak{r}_c}+
\varphi([x,z]_{\mathfrak{r}_c},y)_{\mathfrak{r}_c}+&\\[1mm]
\varphi(x,[y,z]_J)_{\mathfrak{r}_c}-\varphi([x,y]_J,z)_{\mathfrak{r}_c}
+\varphi([x,z]_J,y)_{\mathfrak{r}_c}=0,&\\[1mm]
\end{array}\right.
\end{equation}

\begin{equation}\label{eq222}\left\{
\begin{array}{lll}
[\varphi(x,z)_{\mathfrak{r}_c},y]_J+[x,\varphi(y,z)_J]_J+[x,\varphi(y,z)_{\mathfrak{r}_c}]_J+
[\varphi(x,z)_J,y]_J+&\\[1mm]
\varphi(x,[y,z]_{\mathfrak{r}_c})_J-[\varphi(x,y)_{\mathfrak{r}_c},z]_J-
[\varphi(x,y)_J,z]_J+\varphi(x,[y,z]_J)_J-&\\[1mm]
\varphi([x,y]_{\mathfrak{r}_c},z)_J-\varphi([x,y]_J,z)_J+
\varphi([x,z]_{\mathfrak{r}_c},y)_J+\varphi([x,z]_J,y)_J=0.
\end{array}\right.
\end{equation}

Note that the first six terms of the equality (\ref{eq111}) define a Leibniz $2$-cocycle for the quotient Lie algebra $\mathfrak{r}_c$. Therefore, Leibniz $2$-cocycles of the Lie algebra $\mathfrak{r}_c$ with its trivial extensions on domains $J \otimes R, R\otimes J, J\otimes J$ are included into ${\rm ZL}^2(R,R)$ (the same is true for $2$-coboundaries of the algebra $\mathfrak{r}_c$). Moreover, the last three terms in (\ref{eq111}) appear only for the triples $\{e_1,e_{1},a\}, \ \{e_{1},a,e_1\}, \ \{a,e_{1},e_1\}$ with $a\in \mathfrak{r}_c$.

\begin{prop} \label{prop111} The following $2$-cochains together with a basis of ${\rm ZL}^2(\mathfrak{r}_c,\mathfrak{r}_c)$:
$$\varphi_1(e_1,e_1)=h, \quad \varphi_2(x_1,e_1)=\varphi_2(e_1,x_1)=h, \quad \varphi_{3}(x_1,x_1)=h, \quad \varphi_{4}(x_3,x_1)=h,$$

$$\begin{array}{ll}\left\{  \varphi_{5}(f_1,x_1)=-2h,\right. & \varphi_{5}(f_1,x_3)=-\varphi_{11}(x_3,f_1)=h,\\[4mm]
\end{array}$$
$$
\begin{array}{ll}\left\{
\begin{array}{lll}
\varphi_{6}(x_2,e_2)=-\varphi_{6}(e_2,x_2)=h,&\\
\varphi_{6}(e_2,x_1)=h,&\\[1mm]
\end{array}\right.&
\left\{
\begin{array}{lll}
\varphi_{7}(f_1,x_3)=-\varphi_{7}(x_3,f_1)=h,&\\
\varphi_{7}(f_1,x_1)=(-2)h,& \\[1mm]
\end{array}\right.\\[5mm]
\left\{
\begin{array}{lll}
\varphi_8(e_1,e_1)=\frac{1}{2}x_3,&\\
\varphi_8(h,x_1)=x_3,&\\
\varphi_8(h,f_i)=-\varphi_8(f_i,h)=\frac{1}{2}f_i,&\\
1\leq i\leq n_2,&
\end{array}\right.&
\left\{
\begin{array}{lll}
\varphi_{9}(e_1,e_1)=\frac{1}{2}x_2,&\\
\varphi_{9}(h,x_1)=x_2,&\\
\varphi_{9}(h,e_i)=-\varphi_{9}(e_i,h)=\frac{1}{2}e_i,&\\
2\leq i\leq n_1+1,&\\[1mm]
\end{array}\right.\\[5mm]
\left\{
\begin{array}{lll}
\varphi_{10}(h,h)=-h,&\\
\varphi_{10}(h,x_1)=x_1,&\\
\varphi_{10}(e_1,e_1)=\frac{1}{2}x_1,&\\
\varphi_{10}(h,e_1)=-\varphi_{10}(e_1,h)=\frac{1}{2} e_1,&\\
\varphi_{10}(h,e_i)=-\varphi_{10}(e_i,h)=\frac{i-2}{2}e_i,&\\
\varphi_{10}(h,f_j)=-\varphi_{10}(f_j,h)=\frac{i-1}{2}f_j,&\\
3\leq i\leq n_1+1, \ 2\leq j\leq n_2,\\[1mm]
\end{array}\right.&
\left\{
\begin{array}{lll}
\varphi_{11}(e_1,e_1)=e_1,&\\
\varphi_{11}(e_1,h)=-\varphi_{11}(h,e_1)=h,&\\
\varphi_{11}(h,x_1)=\varphi_{11}(x_1,h)=e_{1},&\\
\varphi_{11}(h,e_i)=-\varphi_{11}(e_i,h)=e_{i+1},&\\
\varphi_{11}(h,f_j)=-\varphi_{11}(f_j,h)=f_{j+1},&\\
2\leq i\leq n_1, \ 1\leq j\leq n_2-1,&\\[1mm]
\end{array}\right.\\[8mm]
\left\{
\begin{array}{lll}
\varphi^{12}_{j}(e_1,e_1)=e_j,& \\
\varphi^{12}_{j}(e_1,h)=-\varphi^{12}_{j}(h,e_1)=e_{j+1},&\\
\varphi^{12}_{j}(x_{1},h)=(j-2) e_j,&\\
\varphi^{12}_{j}(h,x_1)=-(j-4)e_j,&\\
\varphi^{12}_{j}(x_2,h)=-\varphi^{12}_{j}(h,x_2)=e_j,&\\
 2\leq j\leq n_1+1,&\\[1mm]
\end{array}\right.&
\left\{
\begin{array}{lll}
\varphi^{13}_{j}(e_1,e_1)=f_j,&\\
\varphi^{13}_{j}(e_1,h)=- \varphi^{13}_{j}(h,e_1)=f_{j+1},&\\
\varphi^{13}_{j}(x_{1},h)=(j-1) f_j,&\\
\varphi^{13}_{j}(h,x_{1})=(3-j) f_j,&\\
\varphi^{13}_{j}(x_3,h)=-\varphi^{13}_{j}(h,x_3)=f_j,&\\
1\leq j\leq n_2,&\\[1mm]
\end{array}\right.\\[5mm]
\left\{
\begin{array}{lll}
\varphi^{14}_{j}(e_1,e_j)=-\varphi^{14}_{j}(e_j,e_1)=h,&\\
\varphi^{14}_{j}(x_{1},e_{j+1})=(j-1) h,&\\
\varphi^{14}_{j}(e_{j+1},x_1)=(3-j)h,&\\
\varphi^{14}_{j}(x_2,e_{j+1})=-\varphi^{14}_{j}(e_{j+1},x_2)=h,&\\
3\leq j\leq n_1,&\\[1mm]
\end{array}\right.&\left\{
\begin{array}{lll}
\varphi^{15}_{j}(e_1,f_j)=-\varphi^{15}_{j}(f_j,e_1)=h,&\\
\varphi^{15}_{j}(x_{1},f_{j+1})=\varphi^{15}_{j}(f_{j+1},x_1)=j h,&\\
\varphi^{15}_{j}(x_3,f_{j+1})=-\varphi^{15}_{j}(f_{j+1},x_3)=h,&\\
 1\leq j\leq n_2-1,&\\[1mm]
\end{array}\right.
\end{array}$$
form a basis of spaces ${\rm ZL}^2(R,R)$ and ${\rm BL}^2(R,R).$
\end{prop}
\begin{proof}
The proof of this proposition is carried out by straightforward calculations of (\ref{eq4}) and (\ref{eq5}) by using result of Theorem \ref{thmAncochea}. In fact, due to Remark \ref{rem1} and centerlessness of the Lie algebra
$\mathfrak{r}_c$ we conclude that
${\rm H}^2(\mathfrak{r}_c,\mathfrak{r}_c)={\rm HL}^2(\mathfrak{r}_c,\mathfrak{r}_c)$, that is,
${\rm ZL}^2(\mathfrak{r}_c,\mathfrak{r}_c)={\rm BL}^2(\mathfrak{r}_c,\mathfrak{r}_c)$. Taking into account that ${\rm ZL}^2(\mathfrak{r}_c,\mathfrak{r}_c)$ is isomorphically embedded into ${\rm ZL}^2(R,R)$ (respectively, ${\rm BL}^2(\mathfrak{r}_c,\mathfrak{r}_c)$ is isomorphically embedded into ${\rm BL}^2(R,R)$) we need to find a basis of complementary subspaces to ${\rm ZL}^2(\mathfrak{r}_c,\mathfrak{r}_c)$ (respectively, to ${\rm BL}^2(\mathfrak{r}_c,\mathfrak{r}_c)$).

Further, we consider the equalities $(d^2\varphi)(x,y,z)=0$ for the following cases:
$$\begin{array}{llll}
x,y,z\in J, & x\in \mathfrak{r}_c, y,z\in J, & x, z\in J, y\in \mathfrak{r}_c, & x,y\in J, z\in \mathfrak{r}_c,\\
x,y\in \mathfrak{r}_c, z\in J, &  x,z\in \mathfrak{r}_c, y\in J & x\in J, y,z\in \mathfrak{r}_c,
\end{array}$$
from where we get the relations similar to the equations (\ref{eq111}) and (\ref{eq222}). In addition, calculations of (\ref{eq111}) for the triples
$\{e_1,e_{1},a\}, \ \{e_{1},a,e_1\}, \ \{a,e_{1},e_1\} \ \mbox{with} \ a\in \mathfrak{r}_c$
and (\ref{eq222}) for $x,y,z\in \mathfrak{r}_c$ give us some additional relations for complementary subspace to
${\rm ZL}^2(\mathfrak{r}_c,\mathfrak{r}_c)$.

Finally, combining all restrictions on $2$-cocycles and identifying the basis of complementary subspace to ${\rm ZL}^2(\mathfrak{r}_c,\mathfrak{r}_c)$ in ${\rm ZL}^2(R,R)$ we get the required basis of ${\rm ZL}^2(R,R)$.

Applying the same arguments for $2$-coboundaries we complete the proof of theorem.
\end{proof}

\begin{rem} In the above proposition we simplified the calculations using the results for the quotient Lie algebra $\mathfrak{r}_c$. In fact, we exclude calculation of equalities (\ref{eq222}) for the triples $x,y,z\in \mathfrak{r}_c$ except $\{e_1,e_{1},a\}, \ \{e_{1},a,e_1\}, \ \{a,e_{1},e_1\} \ \mbox{with} \ a\in \mathfrak{r}_c.$ Thus, instead of $(\dim  \mathfrak{r}_c)^3$ triples we calculated just $3\dim\mathfrak{r}_c$ triples in (\ref{eq222}).
\end{rem}

As a consequence of Proposition \ref{prop111} we get the following main result.

\begin{thm} The solvable Leibniz algebra $R$ is a cohomologically rigid algebra.
\end{thm}

\section{General case}

In this section we present results similar to obtained in particular case for solvable Leibniz algebras with nilpotent radical $L(\alpha_i, \beta_i), \ 1\leq i\leq k$ and $(k+1)$-dimensional complementary subspace.

Taking into account that the general case is analogous to a special case we omit routine calculations using indexes $n_i$ and induction in the proofs of results below, we just give short sketch their proofs.

The sketch consists of the following steps:

\begin{enumerate}
	\item Firstly, we compute the space $\mathfrak{Der}(L(\alpha_i,\beta_i))$ with $1\leq i\leq k$. Further, we indicate $(k+1)$-pieces nil-independent derivations, which are depends on only non-zero parameters in the diagonal of the general matrix form of derivations.\\
	
	\item Secondly, we construct the solvable Leibniz algebra $\mathbf{R}=L(\alpha_i,\beta_i)\oplus Q$ with 
	$Q=\langle x_1, \ldots, x_{k+1} \rangle$ such that ${{\mathcal{R}_{x_s}}_|}_{L(\alpha_i,\beta_i)}=d_s,$ where $d_s, \ 1 \leq s \leq k+1$ are the nil-independent derivations indicated in the first step. Next, applying the Leibniz identity, the appropriate basis transformations and the mathematical induction we obtain the statement of Theorem \ref{thmDescriptionGeneral}.\\
	
	\item In order to prove the completeness of the solvable Leibniz algebra $\mathbf{R}$ (the first assertion of Theorem \ref{thmCompleteGeneral}) we just need to verify the table of multiplications of $\mathbf{R}$ obtained in the second step and using the fact that any derivation of the quotient Lie algebra $\mathfrak{r}_c=\mathbf{R}/\langle  h\rangle$ is inner together with arguments applied in the proof of the particular case (see Theorem \ref{thmcomplete}) allow us to prove the completeness of the algebra $\mathbf{R}$.\\

\item Finally, in the study of the second cohomology group of the algebra $\mathbf{R}$ we also use the triviality of the second group of cohomologies for the quotient algebra $\mathfrak{r}_c$, that is, we use the equality ${\rm Z}^2(\mathfrak{r}_c,\mathfrak{r}_c)={\rm B}^2(\mathfrak{r}_c,\mathfrak{r}_c)$. By arguments applied in before Proposition \ref{prop111} and due to Remark \ref{rem1} we conclude
$${\rm Z}^2(\mathfrak{r}_c,\mathfrak{r}_c)={\rm ZL}^2(\mathfrak{r}_c,\mathfrak{r}_c) \subseteq {\rm ZL}^2(\mathbf{R},\mathbf{R}), \quad {\rm B}^2(\mathfrak{r}_c,\mathfrak{r}_c)={\rm BL}^2(\mathfrak{r}_c,\mathfrak{r}_c)\subseteq {\rm BL}^2(\mathbf{R},\mathbf{R})$$
we only need to compute the dimensions of complementary  subspaces to ${\rm ZL}^2(\mathfrak{r}_c,\mathfrak{r}_c)$ (respectively, to ${\rm BL}^2(\mathfrak{r}_c,\mathfrak{r}_c)$) in ${\rm ZL}^2(\mathbf{R},\mathbf{R})$ (respectively, in ${\rm BL}^2(\mathbf{R},\mathbf{R})$). Thus, the proof of triviality of the second cohomology group for the algebra $\mathbf{R}$ with coefficient itself is completed by computations of dimensions of the mentioned complementary subspaces.
\end{enumerate}

\begin{thm} \label{thmDescriptionGeneral} Solvable Leibniz algebra with nilpotent radical $L(\alpha_i, \beta_i), \ 1\leq i\leq k$ and $(k+1)$-dimensional complementary subspace is isomorphic to the algebra:
$$\mathbf{R}: \left\{ \begin{array}{lll}
[e_{1},e_1]=h, \quad [h,x_1]=2h,&\\[1mm]
[e_i,e_1]=-[e_1,e_i]=e_{i+1}, &2 \leq i \leq n_1, &\\[1mm]
[e_{n_1+\ldots+n_{j}+i},e_1]=-[e_1,e_{n_1+\ldots+n_{j}+i}]=e_{n_1+\ldots+n_{j}+1+i},& 2\leq i\leq n_{j+1}, \\[1mm]
[e_1,x_1]=-[x_1,e_1]=e_1, &\\[1mm]
[e_i,x_1]=-[x_1,e_i]=(i-2)e_i,& 3\leq i \leq n_1+1,\\[1mm]
[e_{n_1+\ldots+n_{j}+i},x_1]=-[x_1, e_{n_1+\ldots+n_{j}+i}]=(i-2)e_{n_1+\ldots+n_{j}+i} & 2\le i\leq n_{j+1},\\[1mm]
[e_i,x_{2}]=-[x_{2}, e_i]=e_{i}, & 2\le i\leq n_1+1,\\[1mm]
[e_{n_1+\ldots+n_{j}+i},x_{j+2}]=-[x_{j+2}, e_{n_1+\ldots+n_{j}+i}]=
e_{n_1+\ldots+n_{j}+i}, & 2\le i\leq n_{j+1}.\\[1mm]
\end{array}\right.$$
where $1\leq j\leq k-1.$
\end{thm}

\begin{thm} \label{thmCompleteGeneral} The solvable Leibniz algebra $\mathbf{R}$ is complete and its second group of cohomologies in coefficient itself is trivial.
\end{thm}

From the results of the paper \cite{Bal} we obtain rigidity of the algebra $\mathbf{R}$.
\begin{cor} The solvable Leibniz algebra $\mathbf{R}$ is rigid.
\end{cor}

\begin{rem} Note that the structure of the rigid algebra $\mathbf{R}$ depends on the given decreasing sequence $(n_1, n_2, \ldots, n_k).$ Set $p(n)$ the number of such sequences, that is, $p(x)$ is
the number of integer solutions of the equation $n_1 + n_2 + \ldots + n_k = n$ with $n_1 \geq n_2 \geq \ldots \geq n_k \geq 0.$ The asymptotic value of $p(n)$, given in \cite{Hall} by the expression $p(n)\approx\frac{1}{4n\sqrt{3}} e^{\pi\sqrt{2n/3}},$ (where $a(n)\approx b(n)$ means that $\lim\limits_{n\rightarrow\infty}\frac{a(n)}{b(n)}=1$) get the existence of at least $p(n)$ irreducible components of the variety of Leibniz algebras of dimension $n+k+3$.

\end{rem}


\newpage
\begin{thebibliography}{99.}



\bibitem{Lisa}
Adashev J., Camacho L.,  Omirov B.,
{\it Central extensions of null-filiform and naturally graded filiform non-Lie Leibniz algebras},
Journal of Algebra,  479 (2017), 461--486.


\bibitem{Ancochea} Ancochea Berm\'{u}dez J. M., Campoamor-Stursberg R.,
{\it Cohomologically rigid solvable Lie algebras with a nilradical of arbitrary characteristic sequence},
Linear Algebra and its Applications, 488 (2016), 135--147.

\bibitem{Ayupov} Ayupov Sh. A., Omirov B. A.,
{\it  On Leibniz algebras},
Algebra and operator theory, Proceedings of the Colloquium in Tashkent 1997. Kluwer Academic Publishers, 1998, 1--12.

\bibitem{Bal} Balavoine D.,
{\it D\'{e}formations et rigidit\'{e} g\'{e}om\'{e}trique des alg\`{e}bres de Leibniz},
Communications in Algebra, 24 (1996), 3, 1017--1034.

\bibitem{Lindsey} Bosko-Dunbar L., Dunbar J. D., Hird J. T., Stagg K.,
{\it Solvable Leibniz Algebras with Heisenberg Nilradical},
Communications in Algebra, 43 (2015), 6, 2272--2281.

\bibitem{BoPaPo} Boyko V., Patera J., Popovych R.,
{\it Invariants of solvable {L}ie  algebras with triangular nilradicals and diagonal nilindependent elements},
Linear Algebra and its Applications, 428 (2008), 834--854.


\bibitem{Barnes} Barnes D. W.,
{\it On Levi's theorem for Leibniz algebras},
Bulletin of the Australian Mathematical Society, 86 (2012), 2,  184--185.



\bibitem{hei}
Calder\'{o}n A. J., Camacho L. M., Omirov B. A.,
{\it Leibniz algebras of Heisenberg type},
Journal of Algebra, 452 (2016), 427--447.


\bibitem{Cam} Campoamor-Stursberg R.,
{\it Solvable {L}ie algebras with an {$\mathbb{N}$}-graded nilradical of maximal nilpotency degree and their invariants},
Journal of Physics A, 43 (2010), 145202.


\bibitem{Nulfilrad} Casas J. M., Ladra M., Omirov B. A., Karimjanov I. A.,
{\it Classification of solvable Leibniz algebras with null-filiform nilradical},
Linear and Multilinear Algebra, 61 (2013), 6, 758--774.

\bibitem{qua}
Dherin B.,  Wagemann F.,
{\it Deformation quantization of Leibniz algebras},
Advances in Mathematics, 270 (2015), 21--48.


\bibitem{ed2}
Edalatzadeh B., Hosseini S.,
{\it Characterizing nilpotent Leibniz algebras by a new bound on their second homologies},
Journal of Algebra, 511 (2018), 486�498.


\bibitem{ed1}
Edalatzadeh B.,  Pourghobadian P.,
{\it Leibniz algebras with small derived ideal},
Journal of Algebra,  501 (2018), 215--224

\bibitem{Alice} Fialowski A., Magnin L., Mandal A.,
{\it About Leibniz cohomology and deformations of Lie algebras},
Journal of Algebra, 383 (2013),  63--77.

\bibitem{AlbAyupov1}
Gomez-Vidal S.,  Khudoyberdiyev A., Omirov B.,
{\it Some remarks on semisimple Leibniz algebras},
Journal of Algebra, 410 (2014), 526--540.

\bibitem{Gorbat} Gorbatsevich V. V.,
{\it On the Lieification of Leibniz algebras and its applications},
Russian Math. (Iz. VUZ), 60 (2016),  4, 10--16.

\bibitem{Hall} Hall M., {\it Combinatorial Theory}, John Wiley $\&$ Sons Inc, 1986.

\bibitem{geo}
Ismailov N., Kaygorodov I., Volkov Yu.,
{\it The geometric classification of Leibniz algebras},
International Journal of Mathematics, 29 (2018), 5, 1850035.

\bibitem{Iqbol} Karimjanov I. A., Khudoyberdiyev A. Kh., Omirov B. A.,
{\it Solvable Leibniz algebras with triangular nilradicals},
Linear Algebra and its Applications, 466 (2015), 530--546.



\bibitem{deg}
Kaygorodov I., Popov Yu., Pozhidaev A., Volkov Yu.,
{\it Degenerations of Zinbiel and nilpotent Leibniz algebras},
Linear and  Multilinear Algebra, 66 (2018), 4, 704--716.

\bibitem{Khal} Khalkulova Kh. A., Abdurasulov K. K.,
{\it Solvable Lie algebras with maximal dimension of complementary space to nilradical},
Uzbek Mathematical Journal, 2018, 1, 90--98.



\bibitem{hnn}
Ladra M., Shahryari M., Zargeh C.,
{\it HNN-extensions of Leibniz algebras},
Journal of Algebra, 532 (2019), 183--200.



\bibitem{Lod} Loday J.-L.,
{\it Une version non commutative des alg\`ebres de {L}ie: les alg\`ebres de {L}eibniz},
L'Enseignement Mathematique, 39 (1993), 2, 269--293.

\bibitem{Lod-Pir}  Loday J.-L.,  Pirashvili T.,
{\it Leibniz representations of Lie  algebras},
Journal of Algebra, 181 (1996), 2, 414--425.

\bibitem{Mub} Mubarakzjanov G. M.,
{\it  On solvable {L}ie algebras} (Russian),
Izv. Vys\v s. U\v cehn.  Zaved. Matematika, 32 (1963), 1, 114--123.



\bibitem{Snobl}  \u{S}nobl L., Winternitz P.,
{\it Classification and Identification of Lie Algebras},
CRM Monograph series, Centre de Recherches Math\'{e}matiques Montreal, 33, 2014.








%
%
%
%
%


\end{thebibliography}
\end{document}